\documentclass{article}
\usepackage{amsmath,amssymb,amsthm,amsfonts}
\makeatletter
\renewcommand\section{\@startsection {section}{1}{\z@}%
                                   {-3.5ex \@plus -1ex \@minus -.2ex}%
                                   {2.3ex \@plus.2ex}%
                                   {\normalfont\Large\bfseries
                                    \setcounter{equation}{0}
                                    \setcounter{thm}{0}}}
\makeatother
\newtheorem{thm}{Theorem}

\newtheorem{defn}[thm]{Definition}
\newtheorem{rem}[thm]{Remark}
\newtheorem{exa}[thm]{Example}

\newcommand{\q}{\quad}
\newcommand{\qq}{\qquad}

\newcommand{\CommaBin}{\mathbin{\raisebox{0.5ex}{,}}}
\title{\textbf{Stability of a Volterra Integral Equation on Time Scales}}

\author{\small \textbf{Alaa E.\ Hamza$^{1}$}, \textbf{and Ahmed G.~Ghallab$^{2}$}  \\
\small$^\textbf{{1}}$ Department of Mathematics, Faculty of Science, Cairo University, Giza, Egypt.\\
\small E-mail: hamzaaeg2003@yahoo.com\\
\small$^\textbf{{2}}$ Department of Mathematics, Faculty of  Science,\\
 \small {Fayoum University, Fayoum, Egypt.}\\
 \small E-mail: agg00@fayoum.edu.eg}
\date{}

\begin{document}
 \thispagestyle{empty} \maketitle
 \begin{abstract}
 In this paper, we study Hyers-Ulam stability for integral equation of Volterra type in time scale setting. Moreover we study the stability of the considered equation in Hyers-Ulam-Rassias sense. Our technique depends on successive approximation method, and we use time scale variant of induction principle to show that \eqref{eq100} is stable on unbounded domains in Hyers-Ulam-Rassias sense.
 \end{abstract}
 \section{Introduction}
In 1940, S. M. Ulam gave a wide range of talks at the Mathematics Club of
the University of Wisconsin, in which he discussed a number of important unsolved
problems. One of them was the following question:

Let $G_1$ be a group and let $G_2$ be a group endowed with a metric $d$. Given $\epsilon > 0$, does there exist a $\delta > 0$ such that if a mapping $h:G_1\rightarrow G_2 $ satisfies the inequality
$$
d(h(xy),h(x)h(y))<\delta,
$$
for all $x,y\in G_1$, can we find a homomorphism $\theta:G_1\rightarrow G_2$ such that
$$
d(h(x),\theta(x))< \epsilon,
$$
for all $x\in G_1$?

This problem was solved by Hyers for approximately additive mappings on Banach spaces \cite{hyers}. Rassias generalized, in his work \cite{rass}, the result obtained by Hyers. Since then the stability of many functional, differential, integral equations have been investigated, see \cite{gach}, \cite{andras}, \cite{jung}, and references there in.

In this paper we shall consider the non-homogeneous volterra integral equation of the first kind
\begin{equation}\label{eq100}
  x(t)=f(t)+\int_a^tk(t,s)x(s)\Delta s, \qq t\in I_\mathbb{T}:=[a,b]_\mathbb{T},
\end{equation}
where $f\in C_{rd}(I_\mathbb{T},\mathbb{R})$, $k\in C_{rd}(I_\mathbb{T}\times I_\mathbb{T},\mathbb{R})$ and $x$ is the unknown function.

First, we introduce the basic definitions that will be used through out this paper.
\begin{defn}
The integral equation \eqref{eq100} is said to be has\emph{ Hyers-Ulam stability} on $I_\mathbb{T}$ if for any $\varepsilon>0$ and each $\psi\in C_{rd}(I_\mathbb{T},\mathbb{R})$ satisfying
$$
|\psi(t)-f(t)-\int_a^tk(t,s)\psi(s)\Delta s|<\varepsilon, \q \forall \ t \in I_\mathbb{T};
$$
then there exists a solution $\varphi$ of equation \eqref{eq100} and a constant $C\geq0$ such that
$$
|\varphi(t)-\psi(t)|\leq C\,\varepsilon, \q \forall \ t \in I_\mathbb{T}.
$$
\end{defn}
The constant $C$ is called Hyers-Ulam stability constant for equation \eqref{eq100}.
\begin{defn}
The integral equation \eqref{eq100} is said to be has\emph{ Hyers-Ulam-Rassias stability}, with respect to $\omega$, on $I_\mathbb{T}$ if for each  $\psi\in C_{rd}(I_\mathbb{T},\mathbb{R})$ satisfying
$$
|\psi(t)-f(t)-\int_a^tk(t,s)\psi(s)\Delta s|<\omega(t), \q \forall \ t \in I_\mathbb{T};
$$
for some fixed $\omega\in C_{rd}(I_\mathbb{T},[0,\infty))$, then there exists a solution $\varphi$ of equation \eqref{eq100} and a constant $C>0$  such that
$$
|\varphi(t)-\psi(t)|\leq C \omega(t),\q \forall \ t \in I_\mathbb{T}.
$$
\end{defn}

 we shall investigate Hyers-Ulam stability and Hyers-Ulam-Rassias stability of integral equation \eqref{eq100} on both bounded and unbounded time scales intervals.

\section{Hyers-Ulam stability}
In this section we investigate Hyers-Ulam stability of equation on $I_\mathbb{T}:=[a,b]_\mathbb{T}$ by using iterative technique.
\begin{thm}\label{hyers}
The integral equation \eqref{eq100} has Hyers-Ulam stability on $I_\mathbb{T}:=[a,b]_\mathbb{T}$.
\end{thm}
\begin{proof}
For given $\varepsilon>0$ and each $\psi\in C_{rd}(I_\mathbb{T}, \mathbb{R})$ satisfying
$$|\psi(t)-f(t)-\int_a^tk(t,s)\psi(s)\Delta s|<\varepsilon, \q \forall \ t \in I_\mathbb{T},$$
we consider the recurrence relation
\begin{equation}\label{rec}
\psi_{n}(t):=f(t)+\int_a^tk(t,s)\psi_{n-1}(s)\Delta s, \q n=1,2,3,\ldots
\end{equation}
for $t\in I_\mathbb{T}$ with $\psi_0(t)=\psi(t)$. We prove that $\{\psi_{n}(t)\}_{n\in\mathbb{N}}$ converges uniformly to the unique solution of Equation \eqref{eq100} on $I_\mathbb{T}$. We write $\psi_{n}(t)$ as a telescoping sum
$$
\psi_n(t) = \psi _0(t)+\sum_{i=1}^{n} [\psi_{i}(t)-\psi_{i-1}(t)],
$$
 so
\begin{equation}\label{eq6}
 \lim_{n\rightarrow\infty}\psi_n(t) = \psi _0(t)+\sum_{i=1}^{\infty} [\psi_{i}(t)-\psi_{i-1}(t)], \q \forall\ t\in I_\mathbb{T}.
\end{equation}

Using mathematical induction we prove the following estimate
\begin{equation}\label{rr}
|\psi_{i}(t)-\psi_{i-1}(t)|\leq \varepsilon\,M^{i-1}\frac{(t-a)^{i-1}}{(i-1)\,!}\CommaBin \q \forall \ t\in I_\mathbb{T}.
\end{equation}
For $i=1$ we have
$$|\psi_1(t)-\psi(t)|<\varepsilon.$$
So the estimate \eqref{rr} holds for $i=1$. Assume that the estimate \eqref{rr} is true for $i=n\geq1$. We have
\begin{align*}
|\psi_{n+1}(t)-\psi_{n}(t)| & \leq\int_a^t|k(t,s)||\psi_{n}(s)-\psi_{n-1}(s)|\Delta s \\
&\leq M  \int_a^t\varepsilon\,M^{n-1}\frac{(s-a)^{n-1}}{(n-1)\,!}ds\\
&\leq \varepsilon\,M^{n}\frac{(t-a)^{n}}{n\,!}\CommaBin
\end{align*}
hence the estimate \eqref{rr} it valid for $i=n+1$. This shows that the estimate \eqref{rr} is true for all $i\geq1$ on $I_\mathbb{T}$.

See that
 \begin{align*}
   |\psi_{i}(t)-\psi_{i-1}(t)| &\leq \varepsilon\,M^{i-1}\frac{(t-a)^{i-1}}{(i-1)\,!} \\
    & \leq \varepsilon\,M^{i-1}\frac{(b-a)^{i-1}}{(i-1)\,!}\CommaBin
 \end{align*}
and
 $$
 \sum_{i=1}^\infty\varepsilon\,M^{i-1}\frac{(b-a)^{i-1}}{(i-1)\,!}=\sum_{i=0}^\infty \varepsilon\,\frac{[(M(b-a)]^{i}}{i\,!}=\varepsilon\,e^{M(b-a)}.
 $$
Applying Weierstrass M-Test, we conclude that the infinite series
 $$\sum_{i=1}^{\infty} [\psi_{i}(t)-\psi_{i-1}(t)]$$
converges uniformly on $t\in I_\mathbb{T}$. Thus from \eqref{eq6}, the sequence $\{\psi_{n}(t)\}_{n\in\mathbb{N}}$ converges uniformly on $I_\mathbb{T}$ to some $\varphi(t)\in C_{rd}(I_\mathbb{T},\mathbb{R})$. Next, we show that the limit of the sequence $\varphi(t)$ is the exact solution of \eqref{eq1}. For all $t\in I_\mathbb{T}$ and each $n\geq 1$, we have
 $$
 \Big|\int_a^tk(t,s)\psi_{n}(s)-\int_a^tk(t,s)\varphi(s)\Delta s\Big|\leq M\int_a^t|\psi_{n}(s)-\varphi(s)|\Delta s.
 $$
Taking the limits as $n\rightarrow\infty$ we see that the right hand side of the above inequality tends to zero and so
$$
\lim_{n\rightarrow\infty}\int_a^tk(t,s)\psi_{n}(s)\Delta s=\int_a^tk(t,s)\varphi(s)\Delta s, \q \forall \ t \in I_\mathbb{T}.
$$
By letting $n\rightarrow\infty$ on both sides of \eqref{rec}, we conclude that $\varphi(t)$ is the exact solution of \eqref{eq1} on $I_\mathbb{T}$. Then there exists a number $N$ such that $|\psi_N(t)-\varphi(t)|\leq \varepsilon$. Thus
\begin{align*}
|\psi-\varphi|&\leq |\psi(t)-\psi_N(t)|+|\psi_N(t)-\varphi(t)|\\
&\leq |\psi(t)-\psi_1(t)|+|\psi_1(t)-\psi_2(t)|+\cdots +|\psi_{n-1}(t)-\psi_N(t)|+|\psi_N(t)-\varphi(t)|\\
&\leq \sum_{i=1}^N|\psi_{i-1}(t)-\psi_i(t)|+|\psi_N(t)-\varphi(t)|\\
&\leq \sum_{i=1}^N\varepsilon\,M^{i-1}\frac{(b-a)^{i-1}}{(i-1)!}+|\psi_N(t)-\varphi(t)|\\
&\leq \varepsilon\,e^{M(b-a)}+\varepsilon=\varepsilon\,(1+e^{M(b-a)})\,\varepsilon\leq C\,\varepsilon.
\end{align*}
which completes the proof.
\end{proof}

\begin{rem}
We can find an estimate on the difference of two approximate solutions of the integral equation \eqref{eq100}. Let $\psi_1$ and $\psi_2$ are two different approximate solutions to \eqref{eq100} that is for some $\varepsilon_1,\varepsilon_2>0$, and for all $t\in I_\mathbb{T}$
\begin{equation}\label{approx1}
 \Big|\psi_1(t)-f(t)-\int_a^tk(t,s)\psi_1(s)\Delta s\Big|\leq \varepsilon_1,
\end{equation}
 and
\begin{equation}\label{approx2}
  \Big|\psi_2(t)-f(t)-\int_a^tk(t,s)\psi_2(s)\Delta s\Big|\leq \varepsilon_2.
\end{equation}
So
$$
|\psi_1(t)-\psi_2(t)|\leq(\varepsilon_1+\varepsilon_2)e_M(t,a), \q \forall \ t \in I_\mathbb{T}.
$$
If $\psi_1$ is an exact solution of equation \eqref{eq100}, then we have $\varepsilon_1=0$.
\end{rem}
\begin{proof}
Adding the two inequalities \eqref{approx1}, \eqref{approx2} and making use of $|\alpha|-|\beta|\leq|\alpha-\beta|\leq |\alpha|+|\beta|$, we get
$$
\Big|\psi_1(t)-\psi_2(t)-\int_a^tk(t,s)[\psi_1(s)-\psi_2(s)]\Delta s\Big|\leq \varepsilon_1+\varepsilon_2.
$$
$$
|\psi_1(t)-\psi_2(t)|-\Big|\int_a^tk(t,s)[\psi_1(s)-\psi_2(s)]\Big|\leq \varepsilon_1+\varepsilon_2
$$
for all $t\in I_\mathbb{T}$ where $\varepsilon:=\varepsilon_1+\varepsilon_2$.\\
Put
$$
\xi(t):=|\psi_1(t)-\psi_2(t)|, \q \forall \ t\in I_\mathbb{T},
$$
then
 \begin{align*}
   \xi(t) &\leq \varepsilon+\int_a^t|k(t,s)|\xi(s)\Delta s \\
        &\leq \varepsilon+\int_a^tM\xi(s)\Delta s\\
        &\leq \varepsilon+e_M(t,a)\int_a^t\varepsilon \ \frac{M}{e_M(\sigma(s),a)}\Delta s,
 \end{align*}
where we make an application of Gr\"{o}nwall's inequality in the last step. By Theorem we have
$$
\int_a^t\frac{M}{e_M(\sigma(s),a)}\Delta s=-\int_a^t\Big[\frac{1}{e_M(s,a)}\Big]^{\Delta} \Delta s=\Big(1-\frac{1}{e_M(t,a)}\Big),
$$
thus
$$
\xi(t)\leq\varepsilon+\varepsilon \ [e_M(t,a)-1]=\varepsilon \ e_M(t,a),\q \forall \ t\in I_\mathbb{T}.
$$
\end{proof}

\section{Hyers-Ulam-Rassias Stability}
In this section we investigate a result concerning Hyers-Ulam-Rassias stability of equation \eqref{eq100} on both $I_\mathbb{T}:=[a,b]_\mathbb{T}$ and unbounded interval $[a,\infty)_\mathbb{T}$.
\begin{thm}
 Assume $\psi\in C_{rd}(I_\mathbb{T},\mathbb{R})$ satisfying
$$
|\psi(t)-f(t)-\int_a^tk(t,s)\psi(s)\Delta s|<\omega(t), \q \forall \ t \in I_\mathbb{T},
$$
for some fixed $\omega\in C_{rd}(I_\mathbb{T},\mathbb{R}_+)$ for which there exists a constant $P\in(0,1)$ such that
 $$
 \int_a^t\omega(s)\Delta s\leq P\,\omega(t),\q \forall \ t \in I_\mathbb{T}.
 $$
 Then there exist a  unique solution $\varphi$ of Equation \eqref{eq1} such that
 $$
 |\varphi(t)-\psi(t)|\leq \Big(1+\frac{M}{1-P}\Big)\cdot\omega(t), \q \forall \ t \in I_\mathbb{T}.
 $$

\end{thm}

\begin{proof}
Consider the following iterative scheme
\begin{equation}\label{rec2}
\psi_{n}(t):=f(t)+\int_a^tk(t,s)\psi_{n-1}(s)\Delta s, \q n=1,2,3,\ldots
\end{equation}
 for $t\in I_\mathbb{T}$ with $\psi_0(t)=\psi(t)$. By mathematical induction, it is easy to see that the following estimate
 \begin{equation}\label{estim}
   |\psi_{n}(t)-\psi_{n-1}(t)|\leq MP^{\,n-1}\omega(t),
 \end{equation}
 holds for each $n\in \mathbb{N}$ and all $t\in I_\mathbb{T}$. By the same argument as in Theorem \ref{hyers} we prove that
 the sequence ${\psi_{n}(t)}_{n\in\mathbb{N}}$ converges uniformly on $I_\mathbb{T}$ to the unique solution, $\varphi$, of the integral equation \eqref{eq100}. Then there exists a positive integer $N$ such that $|\psi_N(t)-\varphi(t)|\leq w(t),\ t \in I_\mathbb{T}$. Hence
 \begin{align*}
|\psi-\varphi|&\leq |\psi(t)-\psi_N(t)|+|\psi_N(t)-\varphi(t)|\\
&\leq |\psi(t)-\psi_1(t)|+|\psi_1(t)-\psi_2(t)|+\cdots +|\psi_{n-1}(t)-\psi_N(t)|+|\psi_N(t)-\varphi(t)|\\
&\leq \sum_{k=1}^N|\psi_{k-1}(t)-\psi_k(t)|+|\psi_N(t)-\varphi(t)|\\
&\leq \sum_{k=1}^NMP^{\,k-1}\omega(t)+|\psi_N(t)-\varphi(t)|\\
&\leq \sum_{k=1}^NMP^{\,k-1}\omega(t)+|\psi_N(t)-\varphi(t)|\\
&\leq \sum_{k=1}^\infty MP^{\,k-1}\omega(t)+\omega(t)\\
&\leq M\cdot\frac{1}{1-P}\omega(t)+\omega(t)=\Big(1+\frac{M}{1-P}\Big)\omega(t),
\end{align*}
which shows that \eqref{eq100} has Hyers-Ulam-Rassias stability on $I_\mathbb{T}$.
\end{proof}


\begin{thm}
Assume that for a family of statements $A(t), \ t\in [t_0,\infty)_\mathbb{T}$ the following conditions holds
\begin{enumerate}
  \item $A(t_0)$ is true.
  \item for each right-scattered $t\in [t_0,\infty)_\mathbb{T}$ we have $A(t)\Rightarrow A(\sigma(t))$.
  \item for each right-dense $t\in [t_0,\infty)_\mathbb{T}$ there is a neighborhood $U$ such that $A(t)\Rightarrow A(s)$ for all $s\in U, s>t$.
  \item for each left-dense $t\in [t_0,\infty)_\mathbb{T}$ one has $A(s)$ for all $s$ with $s<t\Rightarrow A(t)$.
\end{enumerate}
Then $A(t)$ is true for all $t\in[t_0,\infty)_\mathbb{T}$.
\end{thm}
Next, we prove that the integral equation \eqref{eq100} has Hyers-Ulam-Rassias on unbounded domains.
\begin{thm}
Consider the integral equation \eqref{eq1} with $I_\mathbb{T}:=[a,\infty)_\mathbb{T}$. Let $f\in C_{rd}([a,\infty)_\mathbb{T}, \mathbb{R})$ and $k(t,.)\in C_{rd}([a,\infty)_\mathbb{T}, \mathbb{R})$ for some fixed $t\in [a,\infty)_\mathbb{T}$. Assume $\psi\in C_{rd}(I_\mathbb{T},\mathbb{R})$ satisfying
\begin{equation}\label{eee}
\Big|\psi(t)-f(t)-\int_a^tk(t,s)\psi(s)\Delta s\Big|<\omega(t), \q  t \in I_\mathbb{T};
\end{equation}
where $\omega\in C_{rd}([a,\infty)_\mathbb{T},\mathbb{R}_+)$ with the property
\begin{equation}\label{cond}
  \int_a^t\omega (\tau)\Delta \tau \leq \lambda\,\omega(t), \q \forall \ t \in [a,\infty)_\mathbb{T}.
\end{equation}
for $\lambda\in(0,1)$. Then the integral equation \eqref{eq100} has Hyers-Ulam-Rassias stability, with respect to $\omega$, on $[a,\infty)_\mathbb{T}$.
\end{thm}
\begin{proof}
We apply the time scale mathematical induction in $[a,\infty)_\mathbb{T}$ on the following statements

$A(r):$ the integral equation \eqref{eq100}
$$
x(t)=f(t)+\int_a^tk(t,s)x(s)\Delta s,
$$
has Hyers-Ulam-Rassias stability, with respect to $\omega$, on $[a,r]_\mathbb{T}$.

\textbf{I.} $A(a)$ is trivially true.

\textbf{II.} Let $r$ be a right scattered point and that $A(r)$ holds. That means equation \eqref{eq100} has Hyers-Ulam-Rassias stability, with respect to $\omega$, on $[a,r]_\mathbb{T}$, i.e. for each $\psi:[a,r]_{\mathbb{T}}\rightarrow \mathbb{R}$ satisfying
$$
\Big|\psi(t)-f(t)-\int_a^tk(t,s)\psi(s)\Delta s\Big|<\omega(t), \q  t \in [a,r]_\mathbb{T};
$$
where $\omega\in C_{rd}([a,r]_\mathbb{T},\mathbb{R}_+)$, then there exist a unique solution to equation \eqref{eq100} $\varphi_r:[a,r]_\mathbb{T}\rightarrow \mathbb{R}$ such that
$$
|\varphi_r(t)-\psi(t)|\leq C_1\,\omega(t),\q  t \in [a,r]_\mathbb{T}.
$$
We want to prove that $A(\sigma(r))$ is true. Assume that the function $\psi$ satisfies
$$
\Big|\psi(t)-f(t)-\int_r^tk(t,s)\psi(s)\Delta s\Big|<\omega(t), \q  t \in [r,\sigma(r)]_\mathbb{T}.
$$

Define the mapping $\varphi_{\sigma(r)}:[a,\sigma(r)]_\mathbb{T}\rightarrow \mathbb{R}$ such that
\begin{equation*}
\varphi_{\sigma(r)}(t)=
      \left\{
        \begin{array}{ll}
          \varphi_{r}(t), & \hbox{$t\in [a,r]_\mathbb{T}$;} \\
          f(\sigma(r))+\mu(r)k(\sigma(r),r)\varphi_{r}(r), & \hbox{$t=\sigma(r)$.}
        \end{array}
      \right.
\end{equation*}
It is clear that $\varphi_{\sigma(r)}$ is a solution of \eqref{eq100} on $[a,\sigma(r)]_\mathbb{T}$. Moreover, on we have
\begin{equation*}
|\varphi_{\sigma(r)}(t)-\psi(t)|=
      \left\{
        \begin{array}{ll}
          |\varphi_{r}(t)-\psi(t)|, & \hbox{$t\in [a,r]_\mathbb{T}$;} \\
          |f(\sigma(r))+\mu(r)k(\sigma(r),r)\varphi_{r}(r)-\psi(\sigma(r))|, & \hbox{$t=\sigma(r)$.}
        \end{array}
      \right.
\end{equation*}
See that
\begin{align*}
|\varphi_{\sigma(r)}(\sigma(r))-\psi(\sigma(r))|&=|f(\sigma(r))+\mu(r)k(\sigma(r),r)\varphi_{r}(r)-\mu(r)k(\sigma(r),r)\psi(r)\\
&\ \ \ \ \ \ \ +\mu(r)k(\sigma(r),r)\psi(r)-\psi(\sigma(r))|\\
&\leq |f(\sigma(r))+\mu(r)k(\sigma(r),r)\psi(r)-\psi(\sigma(r))|\\
&\ \ \ \ \ \ \ +|\mu(r)k(\sigma(r),r)||\varphi_{r}(r)-\psi(r)|\\
&\leq \omega(\sigma(r))+MC_1\,\mu(r)\omega(r).
\end{align*}
So we have
\begin{equation*}
|\varphi_{\sigma(r)}(t)-\psi(t)|\leq
     \left\{
       \begin{array}{ll}
         C_1\omega(t), & \hbox{$t\in [a,r]_\mathbb{T}$;} \\
         \omega(\sigma(r))+MC_1\,\mu(r)\omega(r), & \hbox{$t=\sigma(r)$.}
       \end{array}
     \right.
\end{equation*}

\textbf{III.} Let $r\in [a,\infty)_\mathbb{T}$ be right-dense and $U_r$ be a neighborhood of $r$. Assume $A(r)$ is true, i.e. for each $\psi:[a,r]_{\mathbb{T}}\rightarrow \mathbb{R}$ satisfying
$$
\Big|\psi(t)-f(t)-\int_a^tk(t,s)\psi(s)\Delta s\Big|<\omega(t), \q \text{for } t \in [a,r]_\mathbb{T},
$$
where $\omega\in C_{rd}([a,r]_\mathbb{T},\mathbb{R}_+)$, then there exist a unique solution to equation \eqref{eq100} $\varphi_r:[a,r]_{\mathbb{T}}\rightarrow \mathbb{R}$ such that
$$
|\varphi_r(t)-\psi(t)|\leq C_1\,\omega(t),\q \text{for } t \in [a,r]_\mathbb{T}.
$$
We show that $A(\tau)$ is true for all $\tau\in U_r\cap (r,\infty)_{\mathbb{T}}$. For $\tau>r$ assume that the function $\psi$ satisfies
$$
\Big|\psi(t)-f(t)-\int_r^tk(t,s)\psi(s)\Delta s\Big|<\omega(t), \q \text{for  } t \in [r,\tau]_{\mathbb{T}}.
$$
By Theorem for each $\tau\in U_r$, $\tau>r$, the integral equation
$$
x(t)=f(t)+\int_r^tk(t,s)x(s)\Delta s, \q \text{for } t\in [r,\tau]_{\mathbb{T}},
$$
has exactly on solution $\varphi_\tau(\cdot)$.  Therefore the mapping $\xi_\tau:[a,\tau]_{\mathbb{T}}\rightarrow \mathbb{R}$ defined by
\begin{equation*}
  \xi_s(t)=
   \left\{
     \begin{array}{ll}
       \varphi_r(t), & \hbox{$t\in[a,r]_\mathbb{T}$;} \\
       \varphi_\tau(t), & \hbox{$t\in[r,\tau]_\mathbb{T}$.}
     \end{array}
   \right.
\end{equation*}
is a solution of the integral equation
$$
x(t)=f(t)+\int_a^tk(t,s)x(s)\Delta s, \q \text{for } t\in [a,\tau]_{\mathbb{T}}.
$$
We have
  \begin{equation*}
  |\xi_s(t)-\psi(t)|=
   \left\{
     \begin{array}{ll}
       |\varphi_r(t)-\psi(t)|, & \hbox{$t\in[a,r]_{\mathbb{T}}$;} \\
       |\varphi_s(t)-\psi(t)|, & \hbox{$t\in[r,s]_{\mathbb{T}}$.}
     \end{array}
   \right.\\
  \end{equation*}
For $t\in [r,s]_\mathbb{T}$, see that
\begin{align*}
  |\varphi_s(t)-\psi(t)| &= \Big|f(t)+\int_r^tk(t,\tau)\varphi_s(\tau)\Delta \tau \\
    & -\psi(t)+\int_r^tk(t,\tau)\psi(\tau)\Delta \tau-\int_r^tk(t,\tau)\psi(\tau)\Delta \tau\Big|\\
    &\leq |f(t)+\int_r^tk(t,\tau)\psi(\tau)\Delta \tau-\psi(t)|+\int_r^t|k(t,\tau)||\varphi_s(\tau)-\psi(\tau)|\Delta\tau\\
    &\leq  C_1\omega(t)+M\int_r^t\omega(\tau)\Delta\tau\\
    &\leq  C_1\omega(t)+MP\omega(t)=(C_1+MP)\omega(t).
\end{align*}

\textbf{IV.} Let $r\in(a,\infty)_\mathbb{T}$ be left-dense such that $A(s)$ is true for all $s<r$. We prove that $A(r)$ by the same argument as in (\textbf{III}). By the induction principle the statement $A(t)$ holds for all $t\in[a,\infty)_\mathbb{T}$, that means the integral equation \eqref{eq100} has Hyers Ulam Rassias stability on $t\in[a,\infty)_\mathbb{T}$.
\end{proof}

Now we give an example to show that Hyers Ulam stability of volterra Integral equation \eqref{eq100} not necessarily holds on unbounded interval for general time scale.

\begin{exa}
The integral dynamic equation
$$
x(t)=1+5\int_0^tx(s)\Delta s, \q \q t\in [0,\infty)_\mathbb{T},
$$
has exactly one solution $x(t)=e_5(t,0)$, also we have $x(t)=0$ as approximate solution. From Bernoulli's inequality \cite{Boh1}, we have
$$
e_5(t,0)\geq 1+5(t-0),
$$
then we get
$$
\sup_{t\in [0,\infty)}|e_5(t,0)-0|\geq\sup_{t\in [0,\infty)}(1+5t)=\infty.
$$
Hence, there is no Hyers Ulam stability constant.
\end{exa}

\bibliographystyle{srtnumbered}

\end{document}